\documentclass[12pt]{amsart}
\usepackage{amssymb, amsmath}
\usepackage{stmaryrd}

\usepackage[dvipsnames]{xcolor}
\usepackage{color}
\usepackage[utf8]{inputenc}
\usepackage[T1]{fontenc}
\usepackage{fullpage}
\usepackage[normalem]{ulem}
\usepackage{amsmath,amscd,mathrsfs,enumerate}
\usepackage[colorlinks=true, citecolor=blue, linkcolor=red,pagebackref, hyperindex]{hyperref}
\usepackage[all,cmtip]{xy}

\newtheorem{Theorem}{Theorem}[section]
\newtheorem{Theoremx}{Theorem}
\newtheorem{Corollaryx}[Theoremx]{Corollary}

\newtheorem{Potential Theorem}[Theorem]{Potential Theorem}
\newtheorem{Lemma}[Theorem]{Lemma}
\newtheorem{Corollary}[Theorem]{Corollary}

\newtheorem*{Claim*}{Claim}

\theoremstyle{definition}
\newtheorem{Example}[Theorem]{Example}

\newtheorem{Definition}[Theorem]{Definition}
\newtheorem{Question}[Theorem]{Question}

\theoremstyle{remark}
\newtheorem{Remark}[Theorem]{Remark}

\newcommand{\Fnot}{\operatorname{F^\circ}}

\DeclareMathOperator{\Min}{Min}
\DeclareMathOperator{\Ass}{Ass}

\DeclareMathOperator{\Ext}{Ext}

\DeclareMathOperator{\coker}{coker}

\def\p{\mathfrak{p}}

\def\m{\mathfrak{m}}

\def\Z{\mathbb{Z}}

\renewcommand{\geq}{\geqslant}
\renewcommand{\leq}{\leqslant}

\DeclareMathOperator{\ann}{ann}
\newcommand{\ps}[1]{\llbracket {#1} \rrbracket}

\newcommand{\Att}{\operatorname{Att}}
\newcommand{\MinAtt}{\operatorname{MinAtt}}

\newcommand{\soc}{\operatorname{soc}}

\newcommand{\VV}{\mathbb{V}}

\usepackage{comment}

\begin{document}

\title{F-stable secondary representations and deformation of F-injectivity}
\author{Alessandro De Stefani}
\address{Dipartimento di Matematica, Universit{\`a} di Genova, Via Dodecaneso 35, 16146 Genova, Italy}
\email{destefani@dima.unige.it}
\author{Linquan Ma}
\address{Department of Mathematics, Purdue University, West Lafayette, IN 47907, USA}
\email{ma326@purdue.edu}

\thanks{The second author was supported by NSF Grant DMS \#1901672, NSF FRG Grant DMS \#1952366, and a fellowship from the Sloan Foundation.}

\begin{abstract}
We prove that deformation of F-injectivity holds for local rings $(R,\m)$ that admit secondary representations of $H^i_\m(R)$ which are stable under the natural Frobenius action. As a consequence, F-injectivity deforms when $(R,\m)$ is sequentially Cohen--Macaulay (or more generally when all the local cohomology modules $H^i_\m(R)$ have no embedded attached primes). We obtain some additional cases if $R/\m$ is perfect or if $R$ is $\mathbb{N}$-graded.
\end{abstract}

\maketitle

\section{Introduction}

Throughout this article, all rings are commutative, Noetherian, and with multiplicative identity. For rings containing a field of characteristic $p>0$, the seminal work of Hochster and Huneke on tight closure, and subsequent works of many others, has led to a systematic study of the so-called F-singularities. Roughly speaking, these are singularities that can be defined using the Frobenius endomorphism $F$: $R \to R$, which is the map that raises every element of $R$ to its $p$-th power. One of the most studied F-singularities is F-injectivity, which is defined in terms of injectivity of the natural Frobenius actions on the local cohomology modules $H^i_\m(R)$. It was first introduced and studied by Fedder in \cite{Fedder83}.

We say that a property $\mathcal{P}$ of local rings deforms if, whenever $(R,\m)$ is a local ring and $x \in \m$ is a nonzerodivisor such that $R/(x)$ satisfies $\mathcal{P}$, then $R$ satisfies $\mathcal{P}$. While this deformation problem for other classical F-singularities has been settled \cite{Fedder83, HoHu, SinghFReg, SinghFPureFReg}, whether F-injectivity deforms or not in general is still an open question. Fedder proved that F-injectivity deforms when $R$ is Cohen--Macaulay \cite[Theorem 3.4]{Fedder83}, and Horiuchi, Miller, Shimomoto proved that F-injectivity deforms either if $R/(x)$ is F-split \cite[Theorem 4.13]{HMS}, or if $H^i_\m(R/(x))$ has finite length for all $i \ne \dim(R)$ and $R/\m$ is perfect \cite[Theorem 4.7]{HMS}. More recently, the second author and Pham \cite{MaQuy} extended some of these results by further relaxing the assumptions on $R/(x)$.

In this paper, we consider secondary representations of the local cohomology modules $H^i_\m(R)$ (see subsection 2.2 for definitions and basic properties of secondary representations of Artinian modules). It seems natural to ask how the Frobenius action on $H^i_\m(R)$ interacts with a given secondary representation. Our first main result is that F-injectivity deforms when each local cohomology module $H^i_\m(R)$ admits a secondary representation which is stable under the natural Frobenius action (see Definition \ref{Defn F-stable} for details).

\begin{Theoremx}[Theorem \ref{mainThm1}] \label{THMX A}
Let $(R,\m)$ be a $d$-dimensional local ring of characteristic $p>0$ and let $x\in \m$ be a nonzerodivisor on $R$. Suppose for each $i \ne d$, $H^i_\m(R)$ has an F-stable secondary representation. If $R/(x)$ is F-injective, then $R$ is F-injective.
\end{Theoremx}

We prove that secondary components that correspond to minimal attached primes of $H^i_\m(R)$ are always F-stable, see Lemma \ref{Lem.FrobArtinian}. As a consequence, F-injectivity deforms when the attached primes of $H_\m^i(R)$ are all minimal, see Corollary \ref{Corollary minimal} for a slightly stronger statement. In particular, we obtain the following:

\begin{Corollaryx}[Corollary \ref{CorSeqCM}] \label{CORX B}
Let $(R,\m)$ be a $d$-dimensional sequentially Cohen--Macaulay local ring of characteristic $p>0$ and let $x\in \m$ be a nonzerodivisor on $R$. If $R/(x)$ is F-injective, then $R$ is F-injective.
\end{Corollaryx}

We can further relax our assumptions if either the residue field of $R$ is perfect, or if $R$ is $\mathbb{N}$-graded over a field, by only putting conditions on those secondary components of $H^i_\m(R)$ whose attached primes are not equal to $\m$. We refer to Definition \ref{Defn Fnot-stable} for the precise meaning of $\Fnot$-stable secondary representations.

\begin{Theoremx}[Theorem \ref{mainThm2} and Theorem \ref{mainThmGraded}] \label{THMX B}
Let $(R,\m, k)$ be a $d$-dimensional local ring of characteristic $p>0$ that is either local with perfect residue field or $\mathbb{N}$-graded over a field $k$, and let $x\in \m$ be a nonzerodivisor on $R$ (homogeneous in the graded case). Suppose for each $i\ne d$, $H^i_\m(R)$ has an $\Fnot$-stable secondary representation. If $R/(x)$ is F-injective, then $R$ is F-injective.
\end{Theoremx}

\subsection*{Acknowledgments} We thank Pham Hung Quy and Ilya Smirnov for several useful discussions on the topics of this article.

\section{Preliminaries} \label{Section preliminaries}

\subsection{Frobenius actions on local cohomology and F-injectivity}

Let $R$ be a ring of characteristic $p>0$. A Frobenius action on an $R$-module $W$ is an additive map $F$: $W \to W$ such that $F(r\eta) = r^pF(\eta)$ for all $r \in R$ and $\eta \in W$.

Let $I=(f_1,\dots, f_n)$ be an ideal of $R$, then we have the \v{C}ech complex:
$$
C^\bullet(f_1,\dots,f_n; R):= 0\to R\to \oplus_i R_{f_i}\to \cdots\to R_{f_1f_2\cdots f_n}\to 0.
$$
Since the Frobenius endomorphism on $R$ induces the Frobenius endomorphism on all localizations of $R$, it induces a natural Frobenius action on $C^\bullet(f_1,\dots,f_n; R)$, and hence it induces a natural Frobenius action on each $H_I^i(R)$. In particular, there is a natural Frobenius action $F$: $H^i_\m(R) \to H^i_\m(R)$ on each local cohomology module of $R$ supported at a maximal ideal $\m$. A local ring $(R,\m)$ is called F-injective if $F$: $H^i_\m(R) \to H^i_\m(R)$ is injective for all $i$.

\subsection{Secondary representations} We recall some well-known facts on secondary representations that we will use throughout this article. For unexplained facts, or further details, we refer the reader to \cite[Section 7.2]{BrodmannSharp}.

\begin{Definition}
Let $R$ be a ring. An $R$-module $W$ is called secondary if $W \neq 0$ and for each $x \in R$ the multiplication by $x$ map on $W$ is either surjective or nilpotent.
\end{Definition}

One can easily check that, if $W$ is a secondary $R$-module, then $\p = \sqrt{\ann_R(W)}$ is a prime ideal, and $\ann_R(W)$ is $\p$-primary.

\begin{Definition}
Let $R$ be a ring and $W$ be an $R$-module. A secondary representation of $W$ is an expression of $W$ as a sum of secondary submodules, $W = \sum_{i=1}^t W_i$, where each $W_i$ is called a secondary component of this representation.

A secondary representation of $W$ is called irredundant if the prime ideals $\p_i = \sqrt{\ann_R(W_i)}$ are all distinct and none of the summands $W_i$ can be removed from the sum. The set $\{\p_1,\ldots,\p_t\}$ is independent of the irredundant secondary representation and is called the set of attached primes of $W$, denoted by $\Att_R(W)$.
\end{Definition}


Clearly a secondary module has a unique attached prime. Moreover, over a local ring $(R,\m)$, if a nonzero module $W$ has finite length, then $W$ is secondary with $\Att_R(W) = \{\m\}$. A key fact is that every Artinian $R$-module admits an irredundant secondary representation. In particular, all local cohomology modules $H^i_\m(R)$ have an irredundant secondary representation.


\begin{Remark} \label{Remark Ass-Att} When $(R,\m)$ is a complete local ring, Matlis duality induces a correspondence between (irredundant) secondary representations of Artinian modules and (irredundant) primary decompositions of Noetherian modules. In particular, if $(R,\m)$ is complete, and $S$ is an $n$-dimensional regular local ring mapping onto $R$, then $\Att_R(H^i_\m(R)) = \Ass_R(\Ext^{n-i}_S(R,S))$, as the Matlis dual of $H^i_\m(R)$ is isomorphic to $\Ext^{n-i}_S(R,S)$.
\end{Remark}

We conclude this section by recalling the definition of surjective element and strictly filter regular element.

\begin{Definition} \label{Defn surjective} Let $(R,\m)$ be a local ring of dimension $d$. An element $x \in \m$ is called a surjective element if $x \notin \p$ for all $\p\in \bigcup_{i=0}^d \Att_R(H^i_\m(R))$, and $x$ is called a strictly filter regular element if $x \notin \p$ for all $\p\in \left(\bigcup_{i=0}^d \Att_R(H^i_\m(R))\right) \smallsetminus \{\m\}$.
\end{Definition}

\begin{Remark}
\begin{enumerate}
  \item The definition of surjective element we give is not the original one introduced by Horiuchi-Miller-Shimomoto in \cite{HMS}. However, note that $\Ass_R(R)\subseteq \cup_{i=0}^{\dim(R)}\Att_R(H_\m^i(R))$ by \cite[11.3.9]{BrodmannSharp} and thus surjective elements are always nonzerodivisors. Moreover, it follows from the definition that $x$ is a surjective element if and only if $H_\m^i(R)\xrightarrow{\cdot x}H_\m^i(R)$ is surjective for each $i$. Therefore our definition is equivalent to the original definition of surjective element by \cite[Proposition 3.3]{MaQuy}.
  \item The definition of strictly filter regular element was originally introduced by Cuong-Morales-Nhan in \cite{CuongMoralesNhan}. It is easy to see that $x$ is a strictly filter regular element if and only if $\coker(H_\m^i(R)\xrightarrow{x}H_\m^i(R))$ has finite length for each $i$.
\end{enumerate}

\end{Remark}

Surjective elements are important in the study of the deformation problem for F-injectivity. For instance, it was first proved in \cite[Theorem 3.7]{HMS} that if $R/(x)$ is F-injective and $x$ is a surjective element, then $R$ is F-injective (see also \cite[Corollary 3.8]{MaQuy} or the proof of Theorem \ref{mainThm1} in the next section). In fact, we do not know any example that $R/(x)$ is F-injective but $x$ is not a surjective element, see Question \ref{q3}.


\section{F-stable secondary representation}

We introduce the key concept of this article.

\begin{Definition} \label{Defn F-stable} Let $R$ be a ring of characteristic $p>0$, and let $W$ be an $R$-module with a Frobenius action $F$. We say that $W$ admits an F-stable secondary representation if there exists a secondary representation $W = \sum_{i=1}^t W_i$ such that each $W_i$ is F-stable, i.e., $F(W_i) \subseteq W_i$ for all $i$.
\end{Definition}

Observe that, even though we are not explicitly asking that the F-stable secondary representation is irredundant, this can always be achieved, whenever such a representation exists. It seems natural to ask when a secondary component of an Artianina module is F-stable, we show this is always the case for secondary components whose attached primes are minimal in the set of all attached primes.

\begin{Lemma}
\label{Lem.FrobArtinian}
Let $R$ be a ring of characteristic $p>0$, and let $W$ be an Artinian $R$-module with a Frobenius action $F$. Let $W=\sum_{i=1}^t W_i$ be an irredundant secondary representation, with $\p_i=\sqrt{\ann_R(W_i)}$. If $\p_i\in\MinAtt_R(W)$, then $W_i$ is F-stable.

In particular, if $(R,\m)$ is a local ring of characteristic $p>0$ and dimension $d$, then $H_\m^d(R)$ has an F-stable secondary representation.
\end{Lemma}
\begin{proof}
Since $\p_i\in\Min\Att(W)$, we can pick $y\in \cap_{j\neq i}\p_j$ but $y\notin \p_i$. Then $yW_i=W_i$ and $y^NW_j=0$ for all $j\neq i$ and $N\gg0$. Therefore we have $y^NW=W_i$ for all $N\gg0$, and thus $F(W_i)=F(y^NW_i)\subseteq F(y^NW)=y^{pN}F(W)\subseteq y^{pN}W=W_i$.

The last conclusion follows since it is well-known that $\Att_R(H_\m^d(R))=\{\p | \dim(R/\p)=d\}$, see \cite[Theorem 7.3.2]{BrodmannSharp}, in particular, $\Att_R(H_\m^d(R))=\MinAtt_R(H_\m^d(R))$.
\end{proof}

For secondary components whose attached primes are not necessarily minimal, the corresponding secondary components may not be F-stable. However, we do not know whether this can happen when $W$ is a local cohomology module with its natural Frobenius action, see Question \ref{q1}.

\begin{Example}
Let $R=\mathbb{F}_p\ps{x,y}$ and let $W=\mathbb{F}_p\oplus H_\m^2(R)$. Consider the Frobenius action $F$ on $W$ that sends $(1,0)$ to $(1,x^{-p}y^{-1})$ and is the natural one on $H_\m^2(R)$. Then $F$ is injective on $W$, but we claim that $H^2_\m(R)$ is the only proper nontrivial F-stable submodule of $W$. Indeed, let $0 \neq W'$ be an F-stable submodule of $W$, it is enough to show that $0 \oplus H^2_\m(R) \subseteq W'$. Choose $a = (b, c) \neq 0$ inside $W'$. If $c = 0$, then $b \neq 0$. By replacing $a$ with $F(a)$, we can assume that $c \neq 0$. Note that $yF(a) = yF(b, 0) + (0, yF(c)) = (0, yF(c)) \neq 0$ since the action $yF: H^2_\m(R) \to H^2_\m(R)$ is injective. Moreover $H^2_\m(R)$ is simple as an $R$-module with a Frobenius action, so $0 \oplus H^2_\m(R) \subseteq W'$.  Since $W$ is not secondary, this implies that there is no secondary representation of $W$ which is stable with respect to the given Frobenius action (any secondary component with attached prime $\m$ is not F-stable).
\end{Example}

We let $\VV(x)$ denote the set of primes of $R$ which contain $x$. Our first main result is the following.

\begin{Theorem} \label{mainThm1}
Let $(R,\m)$ be a $d$-dimensional local ring of characteristic $p>0$ and let $x\in \m$ be a nonzerodivisor on $R$. Suppose for each $i \ne d$, $H^i_\m(R)$ admits a secondary representation in which the secondary components whose attached primes belong to $\VV(x)$ are F-stable (e.g., $H_\m^i(R)$ has an F-stable secondary representation). If $R/(x)$ is F-injective, then $x$ is a surjective element and $R$ is F-injective.
\end{Theorem}
\begin{proof}
We prove by induction on $i \geq -1$ that multiplication by $x$ is surjective on $H^i_\m(R)$ and that $x^{p^e-1}F^e$ is injective on $H^i_\m(R)$ for all $e>0$. This will conclude the proof, since the first assertion implies $x$ is a surjective element and the second assertion implies $F$ is injective on $H^i_\m(R)$ for all $i$. The base case $i=-1$ is trivial. Suppose both assertions hold for $i-1$; we show them for $i$. Consider the following commutative diagram:
\[\xymatrix{
0 \ar[r]  & H_\m^{i-1}(R/(x)) \ar[d]^{F^e} \ar[r] & H_\m^i(R) \ar[r]^{\cdot x} \ar[d]^{x^{p^e-1}F^e} & H_\m^i(R) \ar[r] \ar[d]^{F^e} & H_\m^i(R/(x)) \ar[d]^{F^e} \ar[r] & \ldots \\
0 \ar[r] & H_\m^{i-1}(R/(x)) \ar[r] & H_\m^i(R) \ar[r]^{\cdot x}  & H_\m^i(R) \ar[r]  & H_\m^i(R/(x)) \ar[r] & \ldots
}\]
where injectivity on the left of the rows follows from our inductive hypotheses. Let $u \in \soc(H^i_\m(R)) \cap \ker(x^{p^e-1}F^e)$. Then $xu=0$, and thus $u$ is the image of an element $v \in H^{i-1}_\m(R/(x))$. Chasing the diagram shows that $F^e(v)=0$. But since $R/(x)$ is F-injective, $F^e$ is injective on $H^{i-1}_\m(R/(x))$ for all $e>0$, so $v=0$ and thus $u=0$. This shows that $x^{p^e-1}F^e$ is injective on $H^i_\m(R)$ for all $e>0$.

It remains to show that multiplication by $x$ is surjective on $H^i_\m(R)$.  
Let $H_\m^i(R)=\sum W_j$ be the secondary representation that satisfies the conditions of the theorem (note that $H_\m^d(R)$ always has an F-stable secondary representation by Lemma \ref{Lem.FrobArtinian}). If there exists $W_j\neq 0$ whose attached prime $\p_j \in \VV(x)$, then it follows from the assumptions that $W_j$ is F-stable. Thus $x^{p^e-1}F^e(W_j)\subseteq x^{p^e-1}W_j = 0$ for all $e \gg 0$ (since $x\in \p_j = \sqrt{\ann_R(W_j)}$). However, we have proved that $x^{p^e-1}F^e$ is injective on $H^i_\m(R)$ for all $e>0$, this implies $W_j=0$ and we arrive at a contradiction. Therefore $x\notin\p$ for all $\p\in \Att_R(H^i_\m(R))$, i.e., multiplication by $x$ is surjective on $H^i_\m(R)$.
\end{proof}

\begin{Corollary} \label{Corollary minimal} Let $(R,\m)$ be a $d$-dimensional local ring of characteristic $p>0$ and let $x\in \m$ be a nonzerodivisor on $R$. Suppose that $\Att_R(H_\m^i(R)) \cap \VV(x)  \subseteq \MinAtt_R(H_\m^i(R))$ for all $i \ne d$ (e.g., when each $H_\m^i(R)$ has no embedded attached primes). If $R/(x)$ is F-injective, then $x$ is a surjective element and $R$ is F-injective.
\end{Corollary}
\begin{proof}
By Lemma \ref{Lem.FrobArtinian}, every irredundant secondary representation of $H_\m^i(R)$ satisfies the assumptions of Theorem \ref{mainThm1} so the conclusion follows.
\end{proof}

We next exhibit an explicit new class of rings for which deformation of F-injectivity holds. Recall that a finitely generated $R$-module $M$ is called sequentially Cohen--Macaulay if there exists a finite filtration $0=M_0\subseteq M_1\subseteq M_2\subseteq \cdots \subseteq M_n=M$ such that each $M_{i+1}/M_i$ is Cohen--Macaulay and $\dim(M_i/M_{i-1})<\dim(M_{i+1}/M_i)$. A local ring $(R,\m)$ is called sequentially Cohen--Macaulay if $R$ is sequentially Cohen--Macaulay as an $R$-module.

\begin{Corollary} \label{CorSeqCM}
Let $(R,\m)$ be a $d$-dimensional sequentially Cohen--Macaulay local ring of characteristic $p>0$ and let $x\in \m$ be a nonzerodivisor on $R$. If $R/(x)$ is F-injective, then $x$ is a surjective element and $R$ is F-injective.
\end{Corollary}
\begin{proof}
First we observe that $R$ is sequentially Cohen--Macaulay implies $\widehat{R}$ is sequentially Cohen--Macaulay and whether $R$ is F-injective (and whether $x$ is a surjective element) is unaffected by passing to the completion. Therefore we may assume $R$ is complete and thus $R$ is a homomorphic image of a regular local ring $S$. By \cite[Theorem 1.4]{HerzogSbarraSequentiallCM}, $R$ is sequentially Cohen--Macaulay is equivalent to saying that, for each $0\leq i \leq d$, $\Ext^{\dim(S)-i} _S(R,S)$ is either zero or Cohen--Macaulay of dimension $i$. In particular, $\Ext^{\dim(S)-i} _S(R,S)$ has no embedded associated primes and hence by Remark \ref{Remark Ass-Att}, $H_\m^i(R)$ has no embedded attached primes for each $0\leq i \leq d$ , that is, $\Att_R(H^i_\m(R)) = \Min\Att_R(H^i_\m(R))$. The conclusion now follows from Corollary~\ref{Corollary minimal}.
\end{proof}

\subsection{Results on local rings with perfect residue field}
If we assume the residue field of $(R,\m)$ is perfect, then we can prove some slight stronger results. The arguments are based on appropriate modifications of the proof of Theorem~\ref{mainThm1}, together with some ideas employed in \cite[Section 5]{MaQuy}. First, we make a modification of the definition of F-stable secondary representation.

\begin{Definition} \label{Defn Fnot-stable}
Let $R$ be a ring of characteristic $p>0$ and $\m$ be a maximal ideal of $R$. Let $W$ be an $R$-module with a Frobenius action $F$. We say that $W$ admits an $\Fnot$-stable secondary representation if there exists a secondary representation $W = \sum_{i=1}^t W_i$ such that $W_i$ is F-stable for all $i$ such that $\Att_R(W_i)\neq \{\m\}$.
\end{Definition}

\begin{Theorem} \label{mainThm2}
Let $(R,\m)$ be a $d$-dimensional local ring of characteristic $p>0$ with perfect residue field, and let $x\in \m$ be a nonzerodivisor on $R$. Suppose for each $i \ne d$, $H^i_\m(R) \ne 0$ admits a secondary representation in which the secondary components whose attached primes belong to $\VV(x)\smallsetminus \{\m\}$ are F-stable (e.g., $H_\m^i(R)$ has an $\Fnot$-stable secondary representation). If $R/(x)$ is F-injective, then $x$ is a strictly filter regular element and $R$ is F-injective.
\end{Theorem}

\begin{proof}
For every $i$, we let $L_i=\coker(H_\m^i(R)\xrightarrow{x}H_\m^i(R))$. We prove by induction on $i \geq -1$ that $L_i$ has finite length and that the Frobenius action $x^{p^e-1}F^e$ on $H_\m^i(R)$ is injective for all $e>0$. This will conclude the proof, since the first assertion implies $x$ is a strictly filter regular element and the second assertion implies $F$ is injective on $H^i_\m(R)$ for all $i$. The initial case $i=-1$ is trivial. Suppose both assertions hold for $i-1$; we show them for $i$. Consider the following commutative diagram:
\[\xymatrix{
0 \ar[r]  & H_\m^{i-1}(R/(x))/L_{i-1} \ar[d]^{F^e} \ar[r] & H_\m^i(R) \ar[r]^{\cdot x} \ar[d]^{x^{p^e-1}F^e} & H_\m^i(R) \ar[r] \ar[d]^{F^e} & H_\m^i(R/(x)) \ar[d]^{F^e} \ar[r] & \ldots \\
0 \ar[r] & H_\m^{i-1}(R/(x))/L_{i-1} \ar[r] & H_\m^i(R) \ar[r]^{\cdot x}  & H_\m^i(R) \ar[r]  & H_\m^i(R/(x)) \ar[r] & \ldots
}\]
Since $L_{i-1}$ has finite length, $F^e$ is injective on $H_\m^{i-1}(R/(x))$ by assumption, and $R/\m$ is perfect, we have that $F^e$ induces a bijection on $L_{i-1}\subseteq H_\m^{i-1}(R/(x))$. Thus, $F^e$ induces an injection on $H_\m^{i-1}(R/(x))/L_{i-1}$ for all $e>0$. Therefore, chasing the diagram above as in the proof of Theorem \ref{mainThm1} we know that $x^{p^e-1}F^e$ is injective on $H_\m^i(R)$ for all $e>0$.

It remains to show that $L_i$ has finite length. 
If there exists $W_j\neq 0$ whose attached prime $\p_j \in \VV(x)\smallsetminus \{\m\}$, then it follows from the assumptions that $W$ is F-stable (note that $H_\m^d(R)$ always has an F-stable secondary representation by Lemma \ref{Lem.FrobArtinian}). Thus $x^{p^e-1}F^e(W_j)\subseteq x^{p^e-1}W_j = 0$ for all $e \gg 0$ (since $x\in \p_j=\sqrt{\ann_R(W_j)}$). However, we have proved that $x^{p^e-1}F^e$ is injective on $H^i_\m(R)$ for all $e>0$, this implies $W_j=0$ and we arrive at a contradiction. Therefore $x\notin \p$ for all $\p\in \Att_R(H^i_\m(R))\smallsetminus \{\m\}$, i.e., $L_i=\coker(H_\m^i(R)\xrightarrow{x}H_\m^i(R))$ has finite length.
\end{proof}

\begin{Corollary}
\label{Corollary minimal perfect residue}
Let $(R,\m)$ be a $d$-dimensional local ring of characteristic $p>0$ with perfect residue field, and let $x\in \m$ be a nonzerodivisor on $R$. Suppose that $\Att_R(H_\m^i(R)) \cap \VV(x)  \subseteq \MinAtt_R(H_\m^i(R)) \cup \{\m\}$ for all $i \ne d$. If $R/(x)$ is F-injective, then $x$ is a strictly filter regular element and $R$ is F-injective. In particular, F-injectivity deforms if $\dim(R/\ann_R(H^i_\m(R))) \leq 1$ for all $i \ne d$ and $R/\m$ is perfect.
\end{Corollary}
\begin{proof}
By Lemma \ref{Lem.FrobArtinian}, every irredundant secondary representation of $H_\m^i(R)$ satisfies the assumptions of Theorem \ref{mainThm2} so the first conclusion follows. To see the second conclusion, it is enough to observe that when $\dim(R/\ann_R(H^i_\m(R))) \leq 1$, we have $\Att_R(H_\m^i(R))\subseteq \MinAtt_R(H_\m^i(R)) \cup \{\m\}$.
\end{proof}

\subsection{Results on $\mathbb{N}$-graded rings} For the rest of this section, we assume that $(R,\m, k)$ is an $\mathbb{N}$-graded algebra over a field $k$ of characteristic $p>0$ ($k$ is not necessarily perfect). Given a graded module $W = \bigoplus_j W_j$ and $a \in \Z$, we denote by $W(a)$ the shift of $W$ by $a$, that is, the graded $R$-module such that $W(a)_j = W_{a+j}$. In this context, when talking about a Frobenius action $F$ on a graded module $W$, we insist that $\deg(F(\eta))=p\cdot\deg(\eta)$ for all homogeneous $\eta\in W$. This is the case for the natural Frobenius action $F$ on the local cohomology modules $H_\m^i(R)$.

The goal of this subsection is to extend Theorem \ref{mainThm2} in this $\mathbb{N}$-graded setting, by removing the assumption that the residue field $k$ is perfect and by strengthening the conclusion to that $x$ is actually a surjective element.


\begin{Theorem} \label{mainThmGraded}
Let $(R,\m,k)$ be a $d$-dimensional $\mathbb{N}$-graded $k$-algebra of characteristic $p>0$ and let $x \in \m$ be a homogeneous nonzerodivisor on $R$. Suppose for each $i \ne d$, $H^i_\m(R)$ admits a secondary representation in which the secondary components whose attached primes belong to $\VV(x)\smallsetminus \{\m\}$ are F-stable (e.g., $H_\m^i(R)$ has an $\Fnot$-stable secondary representation). If $R/(x)$ is F-injective, then $x$ is a surjective element and $R$ is F-injective.
\end{Theorem}
\begin{proof}
Let $\deg(x)=t>0$. We have a graded long exact sequence of local cohomology, induced by the short exact sequence $0 \to R(-t) \xrightarrow{x} R \to R/(x) \to 0$. Moreover, this exact sequence fits in the commutative diagram:
\[\xymatrix{
\ldots \ar[r]  & H_\m^{i-1}(R/(x)) \ar[d]^{F^e} \ar[r] & H_\m^i(R)(-t) \ar[r]^{\cdot x} \ar[d]^{x^{p^e-1}F^e} & H_\m^i(R) \ar[r] \ar[d]^{F^e} & H_\m^i(R/(x)) \ar[d]^{F^e} \ar[r] & \ldots \\
\ldots \ar[r] & H_\m^{i-1}(R/(x)) \ar[r] & H_\m^i(R)(-t) \ar[r]^{\cdot x}  & H_\m^i(R) \ar[r]  & H_\m^i(R/(x)) \ar[r] & \ldots
}\]
Observe that all the Frobenius actions are compatible with the grading. We show by induction on $i \geq -1$ that the map $H^i_\m(R)(-t) \xrightarrow{x} H^i_\m(R)$ is surjective and that $x^{p^e-1}F^e$ is injective on $H^i_\m(R)(-t)$. This will conclude the proof, since the first assertion implies $x$ is a surjective element and the second assertion implies $F$ is injective on $H_\m^i(R)$ for all $i$. The base case $i=-1$ is trivial. Suppose both assertions hold for $i-1$; we show them for $i$. By the same argument as in the proof of Theorem~\ref{mainThm1}, we have that $x^{p^e-1}F^e$ is injective on $H^i_\m(R)(-t)$ for all $e>0$.

It remains to show that multiplication by $x$ map $H^i_\m(R)(-t) \xrightarrow{x} H^i_\m(R)$ is surjective. 
Now by the same argument as in the proof of Theorem \ref{mainThm2}, we know that $L_i=\coker(H_\m^i(R)(-t)\xrightarrow{x}H_\m^i(R))$ has finite length (note that we can ignore the graded structure here). Finally, consider the following commutative diagram:
\[\xymatrix{
0 \ar[r]  & L_i \ar[d]^-{F^e} \ar[r] & H^i_\m(R/(x)) \ar[d]^-{F^e} \ar[r] & H^{i+1}_\m(R)(-t) \ar[d]^-{x^{p^e-1}F^e} \ar[r]^-x & \ldots \\
0 \ar[r]  & L_i \ar[r] & H^i_\m(R/(x)) \ar[r] & H^{i+1}_\m(R)(-t) \ar[r]^-x & \ldots
}\]
Since $F^e$ is injective on $H^i_\m(R/(x))$ by assumption, it is also injective on $L_i$. But since the finite length module $L_i$ is graded and the Frobenius action is compatible with the grading (as the action is induced from $H^i_\m(R/(x))$), this forces $L_i$ to be concentrated in degree zero. If $L_i \ne 0$, then $[L_i]_0 \cong [H^i_\m(R)/xH^i_\m(R)(-t)]_0 \ne 0$, in particular $[H^i_\m(R)]_0 \ne 0$. However, this implies the existence of a nonzero element $u \in [H^i_\m(R)(-t)]_t$. Since we have proved that $x^{p^e-1}F^e$ is injective on $H^i_\m(R)(-t)$, this gives a nonzero element $x^{p^e-1}F^e(u)$ in degree $p^et>0$ for all $e > 0$, which is a contradiction because $[H^i_\m(R)(-t)]_{\gg 0}=0$ (here we are using that the Frobenius action $x^{p^e-1}F^e$ is compatible with the grading on $H_\m^i(R)(-t)$, that is, $\deg(x^{p^e-1}F^e(\eta))=p^e\deg(\eta)$ for all $\eta\in H_\m^i(R)(-t)$). Therefore $L_i=0$, i.e., the multiplication by $x$ map $H^i_\m(R)(-t) \xrightarrow{x} H^i_\m(R)$ is surjective.
\end{proof}

\begin{Corollary} \label{Corollary minimal graded}
Let $(R,\m,k)$ be a $d$-dimensional $\mathbb{N}$-graded $k$-algebra of characteristic $p>0$ and let $x \in \m$ be a homogeneous nonzerodivisor on $R$. Suppose that $\Att_R(H_\m^i(R)) \cap \VV(x)  \subseteq \MinAtt_R(H_\m^i(R)) \cup \{\m\}$ for all $i \ne d$ (e.g., $x$ is a strictly filter regular element). If $R/(x)$ is F-injective, then $x$ is a surjective element and $R$ is F-injective.
\end{Corollary}
\begin{proof}
By Lemma \ref{Lem.FrobArtinian}, every irredundant secondary representation of $H_\m^i(R)$ satisfies the assumptions of Theorem \ref{mainThmGraded} so the conclusion follows.
\end{proof}

\section{Ending questions and Remarks}

We end by collecting some questions that arise from the results in this article. Motivated by Definition \ref{Defn F-stable} and Theorem \ref{mainThm1}, it is natural to ask the following.
\begin{Question} \label{q1}
Let $(R,\m)$ be a local ring of characteristic $p>0$. If $H^i_\m(R) \ne 0$, does it admit an F-stable secondary representation?
\end{Question}

By Theorem \ref{mainThm1}, a positive answer to Question \ref{q1} implies that F-injectivity deforms.

\begin{Question} \label{q1maximal}
Let $(R,\m)$ be a local ring of characteristic $p>0$. If $H^i_\m(R)\ne 0$, does it admit a secondary representation such that the secondary component with attached prime $\m$, if not zero, is F-stable?
\end{Question}

This is weaker than Question \ref{q1}, but an affirmative answer also implies that F-injectivity deforms. Suppose $R/(x)$ is F-injective, we will show $x$ is a surjective element and thus $R$ is F-injective by \cite[Theorem 3.7]{HMS} (or use the same argument as in Theorem \ref{mainThm1}). In fact, if $x\in\p$ for some $\p\in \Att_R(H^i_\m(R))$, then $x\in \p R_\p\in\Att_{R_\p}(H^j_{\p R_\p}(R_\p))$ for some $j$ and $R_\p/xR_\p$ is still F-injective. Now an affirmative answer to Question \ref{q1maximal} applied to $(R_\p, \p R_\p)$ implies that there exists a nonzero secondary component of $H^j_{\p R_\p}(R_\p)$ with attached prime $\p R_\p$ that is F-stable, and we can argue as in the proof of Theorem \ref{mainThm1} to arrive at a contradiction.


\begin{Question} \label{q3}
Let $(R,\m)$ be a local ring of characteristic $p>0$, and let $x \in \m$ be a nonzerodivisor on $R$. If $R/(x)$ is F-injective, is it true that $\m \notin \Att(H^i_\m(R))$ for all $i$?
\end{Question}

Similar to the discussion above, we point out that an affirmative answer to Question \ref{q3} also implies that $x$ is a surjective element (and hence implies that F-injectivity deforms): if not, then $x\in\p$ for some $\p\in \Att_R(H^i_\m(R))$, but then $R_\p/xR_\p$ is still F-injective and $\p R_\p\in\Att_{R_\p}(H^j_{\p R_\p}(R_\p))$ for some $j$, which contradicts Question \ref{q3} for $(R_\p, \p R_\p)$.

\bibliographystyle{alpha}
\bibliography{References}

\begin{thebibliography}{CMN04}

\bibitem[BS13]{BrodmannSharp}
M.~P. Brodmann and R.~Y. Sharp.
\newblock {\em Local cohomology}, volume 136 of {\em Cambridge Studies in
  Advanced Mathematics}.
\newblock Cambridge University Press, Cambridge, second edition, 2013.
\newblock An algebraic introduction with geometric applications.

\bibitem[CMN04]{CuongMoralesNhan}
Nguyen~Tu Cuong, Marcel Morales, and Le~Thanh Nhan.
\newblock The finiteness of certain sets of attached prime ideals and the
  length of generalized fractions.
\newblock {\em J. Pure Appl. Algebra}, 189(1-3):109--121, 2004.

\bibitem[Fed83]{Fedder83}
Richard Fedder.
\newblock {$F$}-purity and rational singularity.
\newblock {\em Trans. Amer. Math. Soc.}, 278(2):461--480, 1983.

\bibitem[HH94]{HoHu}
Melvin Hochster and Craig Huneke.
\newblock {$F$}-regularity, test elements, and smooth base change.
\newblock {\em Trans. Amer. Math. Soc.}, 346(1):1--62, 1994.

\bibitem[HMS14]{HMS}
Jun Horiuchi, Lance~Edward Miller, and Kazuma Shimomoto.
\newblock Deformation of {$F$}-injectivity and local cohomology.
\newblock {\em Indiana Univ. Math. J.}, 63(4):1139--1157, 2014.
\newblock With an appendix by Karl Schwede and Anurag K. Singh.

\bibitem[HS02]{HerzogSbarraSequentiallCM}
J\"{u}rgen Herzog and Enrico Sbarra.
\newblock Sequentially {C}ohen-{M}acaulay modules and local cohomology.
\newblock In {\em Algebra, arithmetic and geometry, {P}art {I}, {II} ({M}umbai,
  2000)}, volume~16 of {\em Tata Inst. Fund. Res. Stud. Math.}, pages 327--340.
  Tata Inst. Fund. Res., Bombay, 2002.

\bibitem[MQ18]{MaQuy}
Linquan Ma and Pham~Hung Quy.
\newblock Frobenius actions on local cohomology modules and deformation.
\newblock {\em Nagoya Math. J.}, 232:55--75, 2018.

\bibitem[Sin99a]{SinghFPureFReg}
Anurag~K. Singh.
\newblock Deformation of {$F$}-purity and {$F$}-regularity.
\newblock {\em J. Pure Appl. Algebra}, 140(2):137--148, 1999.

\bibitem[Sin99b]{SinghFReg}
Anurag~K. Singh.
\newblock {$F$}-regularity does not deform.
\newblock {\em Amer. J. Math.}, 121(4):919--929, 1999.

\end{thebibliography}

\end{document}